\documentclass{amsart}
\usepackage{mathtools}
\usepackage{epstopdf,
}
\usepackage[all]{xy}

\usepackage{listings,color}
\usepackage[dvipsnames]{xcolor}
\usepackage{url}

\usepackage{color}

\newtheorem{thm}{Theorem}[section]

\newtheorem{prop}[thm]{Proposition}
\newtheorem{defn}[thm]{Definition}
\newtheorem{rem}[thm]{Remark}

\newcommand{\CGA}{\operatorname{CGA}}
\newcommand{\PGA}{\operatorname{PGA}}
\newcommand{\bP}{\mathbf{P}}
\newcommand{\e}{\mathbf{e}}
\newcommand{\bell}{\boldsymbol{\ell}}
\newcommand{\bp}{\mathbf{p}}
\newcommand{\bA}{\mathbf{A}}
\newcommand{\bB}{\mathbf{B}}
\newcommand{\bC}{\mathbf{C}}
\newcommand{\bD}{\mathbf{D}}

\newcommand{\bF}{\mathbf{F}}

\newcommand{\bT}{\mathbf{T}}
\newcommand{\bt}{\mathbf{t}}
\newcommand{\bR}{\mathbf{R}}
\newcommand{\bI}{\mathbf{I}}
\newcommand{\bV}{\mathbf{V}}
\newcommand{\bs}{\mathbf{s}}

\begin{document}

%
%
%
%
%
%
%
%
%

\title[PGA in CGA]
{Projective Geometric Algebra as a Subalgebra of Conformal Geometric algebra}
\author{Ale\v{s} N\'{a}vrat}
\address{%
	Institute of Mathematics,  \\ Faculty of Mechanical Engineering,\\
	Brno University of Technology,\\ Czech Republic}

\email{navrat.a@fme.vutbr.cz}

\author{Jaroslav Hrdina}
\address{%
	Institute of Mathematics,  \\ Faculty of Mechanical Engineering,\\
	Brno University of Technology,\\ Czech Republic}

\email{hrdina@fme.vutbr.cz}

\author{Petr Va\v s\' ik}

\address{%
Institute of Mathematics,  \\ Faculty of Mechanical Engineering,\\
Brno University of Technology,\\ Czech Republic}

\email{vasik@fme.vutbr.cz }


\author{Leo Dorst
}
\address{%
	Informatics Institute,  \\ Faculty of Science,\\
	University of Amsterdam,\\ The Netherlands}

\email{l.dorst@uva.nl}

\thanks{The first three authors were
supported by a grant
no.:~FSI-S-20-6187.}




\begin{abstract}
We show that if PGA is understood as a subalgebra of CGA in mathematically correct sense, then the flat objects share the same representation in PGA and CGA. Particularly, we treat duality in PGA. This leads to unification of PGA and CGA objects which is important especially for software implementation and symbolic calculations.
\end{abstract}

\maketitle

\section{Introduction}
Projective geometric algebra (PGA) is the right model for Euclidean geometry and computations with flat objects, see \cite{gunn1,gunn2}. Conformal geometric algebra (CGA) contains the same model and, moreover, allows transformations of round objects and dilation (conformal geometry), see \cite{hild1,hild2, h0,h1}. Clearly, PGA is a subalgebra of CGA but the representation of Euclidean geometry looks very different at the first sight. The PGA representation of a point is a multivector of grade $n-1$ while  a CGA point is of grade 1. This indicates that we have to think dually, or in other words in a complementary way. In what follows, we clarify how PGA can be viewed in CGA. The inclusion leads us to a geometric definition of PGA duality which slightly differs from the definition in literature, compare \cite{3dpga} and Table \ref{PGAdualitytable}. Consequently, a unified approach to both algebras is introduced. We treat the case $n=3$ in this paper, however, the results hold for arbitrary dimension $n$, in particular also for $n=2$.

First, we introduce briefly the frameworks of CGA and PGA as models of Euclidean geometry and we summarize basic formulae in  Section \ref{GAforE}. In Section \ref{PinC}, we show that there are two naturally related copies of PGA in CGA, see Proposition \ref{PGAinCGA}. After the identification of the two copies, duality in PGA is obtained in terms of CGA operations, see Proposition \ref{PGAdualityCGA}. This directly describes the correspondence between flat objects and versors for Euclidean transformations in CGA, and the objects and versors in PGA, see Proposition \ref{GeometricEmbedding}. Basic ideas are then demonstrated on a simple example.

\subsection{Notation}
We denote elements of geometric algebras by bold letters - capitals  for general multivectors and lower case letters for vectors. We also set the notation such that we can easily distinguish the elements of different algebras. Namely, 
$\bA,\bB,\dots$ will denote multivectors in PGA, $\bA_c,\bB_c,\dots$ will denote multivectors in CGA and $\bA_E,\bB_E,\dots$ will denote $\mathbb{G}_3$ elements.
Similar notation will be used for objects and transformations of algebras, namely  $\bP,\bell,\bp$ will denote a point, line and plane in PGA, respectively, while $\bP_c,\bell_c,\bp_c$ will denote direct representations of the respective objects in CGA. The corresponding dual representations will be accented by star superscript.  By $\bP_E$ we mean the representation of a Euclidean point in $\mathbb{G}_3$. The versors for rotations and translations in PGA and CGA will be denoted by $\bR,\bT$ and $\bR_c,\bT_c,$ respectively.
In order to distinguish among dualities in different algebras and to be consistent with the notation in literature we denote the duality in CGA by an ordinary star symbol, by $*_P$ the duality in PGA and by $*_E$ the duality in $\mathbb{G}_3,$ respectively.

\section{Geometric algebras for Euclidean geometry} \label{GAforE}

It is well known that the geometric algebra $\mathbb{G}_3$ is useful for the description of vector geometry and rotations, see e.g. \cite{Dorst} for the geometric background. If we want to describe Euclidean geometry, we need to add a null vector in order to represent translations. The minimal way is to raise the dimension by one and add a one-dimensional space of null vectors. This model is known as the projective geometric algebra (PGA). However, this procedure returns a Clifford algebra with a degenerate quadratic form. Another way is to raise the dimension by two and add a symplectic two--dimensional vector space. Then the quadratic form is nondegenerate with indefinite signature and we have even two linearly independent null vectors. The resulting model is known as the conformal geometric algebra (CGA). We list the basic facts and formulae in both models.

\subsection{Euclidean geometry in CGA}
From the algebraic point of view, CGA is a Clifford algebra defined by a nondegenerate quadratic form of signature $(4,1,0).$  Vectors 
$ 
\e_0,\e_1,\e_2,\e_3,\e_\infty 
$
denote an orthogonal basis of the generating vector space $\mathbb{R}^{4,1}$ with inner product given by the quadratic form
\begin{align} \label{BCGA}
B=\begin{pmatrix}
0&0 &-1\\
0&1_{3 \times 3} & 0\\
-1 & 0&0
\end{pmatrix}.
\end{align} 
Hence $\e_0,\e_\infty$ are null vectors and $\e_0\cdot\e_\infty=\e_\infty\cdot\e_0=-1.$ The duality in CGA is defined by
\begin{align} \label{CGAduality}
\bA_c^*=\bA_c\bI_c^{-1}=\bA_c \cdot \bI_c^{-1},
\end{align}
 where $\bI_c=\e_{0123\infty}$ is the conformal {\it pseudoscalar}, $\bI_c^{-1}=-\bI_c$.
The geometry in CGA is defined by the following embedding of a Euclidean point $\bP_E$ into the geometric algebra $\mathbb{G}_3,$ particularly onto an element $\bP_c$ of the form
\begin{align} \label{CGApoint}
\bP_c = \e_0 + \bP_E +\tfrac12 (\bP_E\cdot\bP_E)\e_\infty,
\end{align}
where $\bP_E\cdot\bP_E$ coincides with the square of Euclidean norm. In coordinates, if $\bP_E=x\e_1+y\e_2+z\e_3$, then we get a well known formula
\begin{align}
\bP_c = \e_0 + x\e_1+y\e_2+z\e_3 +\tfrac12 (x^2+y^2+z^2)\e_\infty
\end{align}
together with the standard property $\bP_c\cdot\bP_c=0.$
The nondegeneracy of the quadratic form \eqref{BCGA} implies that we have two mutually dual representations of geometric objects in CGA. Namely, a multivector $\bA_c$ is the {\it direct representation} (also called OPNS representation) of an object in CGA if and only if the object is formed exactly by points $\bP_c$ satisfying
\begin{align}
\bP_c\wedge\bA_c =0.
\end{align}
The duality in CGA given by \eqref{CGAduality} defines a {\it dual representation} (or IPNS representation). Namely, the same object can be also represented by $\bA_c^*$ in the sense that it is formed exactly by points $\bP_c$ satisfying
\begin{align}
\bP_c\cdot \bA_c^* =0,
\end{align} 
where the dot denotes the {\it inner product}.
Note that this duality of representations follows from the duality between inner and outer product.
The direct representation is useful for constructing geometric objects from points while the advantage of the dual representation is that one can easily read off the internal parameters of the objects.

Taking outer products of points in CGA we get representatives of general spheres spanned by these points, i.e. point pairs, circles, spheres and also flat objects if one of the points lies at infinity. Thus for a flat point $\bF\bP_c$, a line $\bell_c$ spanned by points $\bP_{1c},\bP_{2c}$ and a plane $\bp_c$ spanned by points $\bP_{1c},\bP_{2c},\bP_{3c}$ we have the following respective representations
\begin{align} \label{CGAobjects}
\bF\bP_c&=\bP_c\wedge\e_\infty,\\
\bell_c&=\bP_{1c}\wedge\bP_{2c}\wedge\e_\infty, \\
\bp_c&=\bP_{1c}\wedge\bP_{2c}\wedge\bP_{3c}\wedge\e_\infty.
\end{align}
%

Let us also recall that Euclidean transformations are represented by versors which act on objects by conjugation. The versor for translation by vector $\vec{t}$, which we identify with $\bt_E\in\mathbb{G}_3$, is given by
\begin{align} \label{CGAtranslator}
\bT_c=e^{-\frac12\bt_E\e_\infty}=1-\tfrac12\bt_E\e_\infty.
\end{align}
The versor for rotation by an angle $\alpha$ and with normalised dual representation $\bell^*,$ i.e. $\bell^*\cdot\bell^*=-1,$ is represented by CGA element 
\begin{align} \label{CGArotor}
\bR_c=e^{\frac12 {\alpha}\bell^*}=\cos\tfrac{\alpha}2+\sin\tfrac{\alpha}2\bell^*.
\end{align}

\begin{rem}
	For a general dimension $n$, we consider the orthogonal basis $\e_0,\e_1,\dots,\e_n,\e_\infty\in\mathbb{R}^{n+1,1}$ with inner product given by matrix \eqref{BCGA}, where the middle block is the $n\times n$ identity matrix. The duality prescription \eqref{CGAduality} remains, however, the pseudoscalar satisfies
	\begin{align*}
	\bI_c^{-1}=(-1)^{n(n-1)/2}\bI_c,
	\end{align*}
	thus the duality map is either an involution or an anti--involution depending on the dimension. The equation \eqref{CGApoint} defining a CGA point  is independent of the dimension as well as formulae \eqref{CGAtranslator} and \eqref{CGArotor} for CGA transformations. The representations of Euclidean objects (flats) in CGA are $\bP_{1c}\wedge\cdots\wedge\bP_{kc}\wedge\e_\infty$ for any $1\leq k\leq n.$
\end{rem}


%
%
%

\subsection{Euclidean geometry in PGA}
We shall use conventions and notation as close to \cite{gunn1,gunn2} and \cite{3dpga} as possible, however, we modify the sign conventions for the duality in PGA slightly. The reason for that will be clarified in the sequel. Basically, the reason is to fit the PGA concept to the concept of CGA. We want to stress that these changes do not influence the validity of general formulae in \cite{3dpga} for incidence relations, projections, rejections etc.

Algebraically, PGA is a Clifford algebra generated by a degenerate quadratic form of signature $(3,0,1)$. We consider a basis of the generating vector space $\mathbb{R}^{3,0,1}$ in which the quadratic form is given by matrix 
\begin{align} \label{BPGA}
B=\begin{pmatrix}
0&0\\
0&1_{3 \times 3}\\
\end{pmatrix},
\end{align} 
and we denote this basis by $\e_0,\e_1,\e_2, \e_3$, i.e. $\e_0$ is a null vector. Note that the symbols for basis elements are the same as symbols for the corresponding basis elements in CGA in this notation, which is usual in literature. Indeed, PGA can be viewed as the subalgebra of CGA spanned by these elements. However for relating the geometry of PGA and CGA,  
the $\e_0$ from PGA plays rather the role of the element $\e_\infty$ in the CGA notation. This relation is discussed in detail in the following sections.  

It is well known that the duality in PGA cannot be obtained as the duality in CGA, by the division with the pseudoscalar, because the quadratic form \eqref{BPGA} defining PGA is degenerate. The idea of \cite{gunn1,gunn2} is to use a Hodge duality approach. A similar idea may be found in \cite{Lounesto}. Namely, the dual to a basis element $\bA$ can be defined as the complement to the projective pseudoscalar,  $\bA\wedge \bA^* =\bI$ or in the reverse order. However, such a map is neither involutive nor anti--involutive in general and the signs coming from such definition do not fit in the CGA duality. Therefore we suggest a new definition in the next Section, see Definition \ref{PGAdualityCGA}. Note that a specific duality concept for constant curvature spaces has been introduced in \cite{Lasenby}.

Representation of the Euclidean geometry in PGA is given by the embedding of a point.
In coordinates, an Euclidean point $\bP_E=x\e_1+y\e_2+z\e_3$ is represented in PGA by the multivector of grade three 
\begin{align} 
\bP=x\e_{032}+y\e_{013}+z\e_{021}+\e_{123}.
\end{align}
The formula for a point in PGA can be written in a coordinate free way using the Euclidean duality $*_E$, given by the division with the Euclidean pseudoscalar $\bI_E=\e_{123},$ particularly $\bP=\e_0\bP_E^{*_E}+\bI_E,$ which can be rewritten as
\begin{align} \label{PGApoint}
\bP=\bI_E+\bP_E^{*_E}\e_0,
\end{align}
since the grade of $\bP_E^{*_E}$ is two.
The degeneracy of the PGA quadratic form \eqref{BPGA} causes us to have only one representation of geometric objects in PGA. Since the grade of points is three, therefore sub--maximal, the only way to represent objects  is as the null space of the regressive product (RPNS representation). Recall that the regressive product is dual to the outer product, i.e. 
$
(\bA \vee \bB)^{*}=\bA^* \wedge \bB^*.
$
Hence a point represented by $\bP$ belongs to an object represented by $\bA$ if and only if
\begin{align}
\bP \vee \bA =0.
\end{align}

By regressive products of respective points in PGA we get representatives of flat objects. Thus for a line $\bell$ spanned by points $\bP_{1},\bP_{2}$ and a plane $\bp$ spanned by points $\bP_{1},\bP_{2},\bP_{3}$ we have the representations
\begin{align}
\bell&=\bP_1 \vee \bP_2, \label{PGAline}\\
\bp&=\bP_1 \vee \bP_2 \vee \bP_3. \label{PGAplane}
\end{align}

One can find several formulae for PGA representations of Euclidean transformations in \cite{gunn1}, however a direct formula for translation by vector $\bt_E$ is missing. If one reads between the lines, the versor for translation is given by
\begin{align}
\bT=e^{-\frac12 \e_0\bt_E}=1-\tfrac12\e_0\bt_E.
\end{align}
Note that this formula is a direct consequence of formula \eqref{CGAtranslator} for the translator in CGA and  Proposition \ref{GeometricEmbedding} which will be proved later.
Rotation is realized by \eqref{CGArotor}, the same versor as in CGA, thus $\bR=\bR_c.$ However, the dual line $\bell^*$ in the formula must be replaced by $\bell$ given by \eqref{PGAline} in this case.
\begin{rem}
	For a general dimension $n,$ we consider the orthogonal basis $\e_0,\e_1,\dots,\e_n\in\mathbb{R}^{n,0,1}$ with inner product given by matrix \eqref{BPGA}, where the lower--right block is the $n\times n$ identity matrix. The equation  \eqref{PGApoint} defines a PGA point in any dimension. The representations of transformations are independent of the dimension and the Euclidean objects are given by $\bP_{1}\vee\cdots\vee\bP_{k}$ for any $1\leq k\leq n.$ The duality in $n$D PGA is discussed in the next Section, see Remark \ref{remDual} .
\end{rem}

\section{PGA in CGA} \label{PinC}
Our main finding will be that the usual choice of basis and inner product in CGA, 
gives two distinct subalgebras of CGA which are both isomorphic to PGA and also an involutive isomorphism on CGA that relates these two subalgebras. In such identification, the duality in PGA can be seen as a "twisted" CGA duality. Moreover, if the right subalgebra is chosen, representations of flat objects in CGA and PGA coincide.

\subsection{Duality in PGA}
In CGA, the subalgebra formed by  elements which do not contain $\e_\infty$ is present as well as the subalgebra formed by  elements which do not contain $\e_0.$ The former is generated by $\langle \e_0,\e_1,\e_2,\e_3\rangle\cong\mathbb{R}^{3,0,1}$ and will be denoted by $\CGA_0$ and the latter is generated by $\langle \e_1,\e_2,\e_3,\e_\infty\rangle\cong\mathbb{R}^{3,0,1}$ and will be denoted by $\CGA_\infty$ in the sequel. The notation accents the present null vector and stresses the fact that both are subalgebras in CGA.

 The choice of CGA basis and the inner product defines also a distinct involution in CGA which relates subalgebras $\CGA_0$ and $\CGA_\infty$. Namely, it defines an isomorphism between vector space $\mathbb{R}^{4,1}$ and its dual, by taking the usual dual basis. Quadratic form \eqref{BCGA} defines another isomorphism between these spaces which is known as the musical isomorphism in literature. The composition of these two isomorphisms is a bijective linear map of $\mathbb{R}^{4,1}$ onto itself, for $i=1,2,3$ given by 
\begin{align} \label{sharp}
\sharp: \e_i\mapsto \e_i, 
\e_0\mapsto -\e_\infty, \e_\infty\mapsto -\e_0,
\end{align}
where we used the notation of musical isomorphism in order to distinguish between the duality on vector space $\mathbb{R}^{4,1}$ and the usual CGA duality. Clearly, this map is a linear involution and preserves the quadratic form in CGA,  thus it defines a unique extension to CGA as a homomorphism of Clifford algebras. In the following, we will use the symbol $\sharp$ also for this extension. Using this notation we can summarize the above observations into the following statement.
\begin{prop} \label{PGAinCGA}
	The choice of basis of null vectors $\e_0,\e_\infty$ in CGA defines two subalgebras $CGA_0$ and $CGA_\infty$ isomorphic to PGA and an involutive isomorphism $\sharp$ between them which acts by replacing $\e_0$ with $-\e_\infty$ in each basis blade.
\end{prop}
Let us describe the structure of subalgebras CGA$_0$ and CGA$_\infty$ and the isomorphism $\sharp$ in a more detailed way. Intersection of these subalgebras is generated by $\langle \e_1,\e_2,\e_3\rangle$ with inner product given by  the identity matrix, thus it forms the algebra $\mathbb{G}_3$. The map $\sharp$ that switches between the two subalgebras   acts as the identity on this intersection. On the complement of the union of the subalgebras to CGA, it acts as minus identity. Note that the union of subalgebras contains the elements with either $e_0$ or $e_\infty$ only, therefore the complement to CGA is not empty containing elements with both $e_0$ and $e_\infty$. 
 Schematically, this is depicted in Figure \ref{2copies}.
\begin{figure}[h] \label{2copies}
 	\includegraphics[scale=0.45]{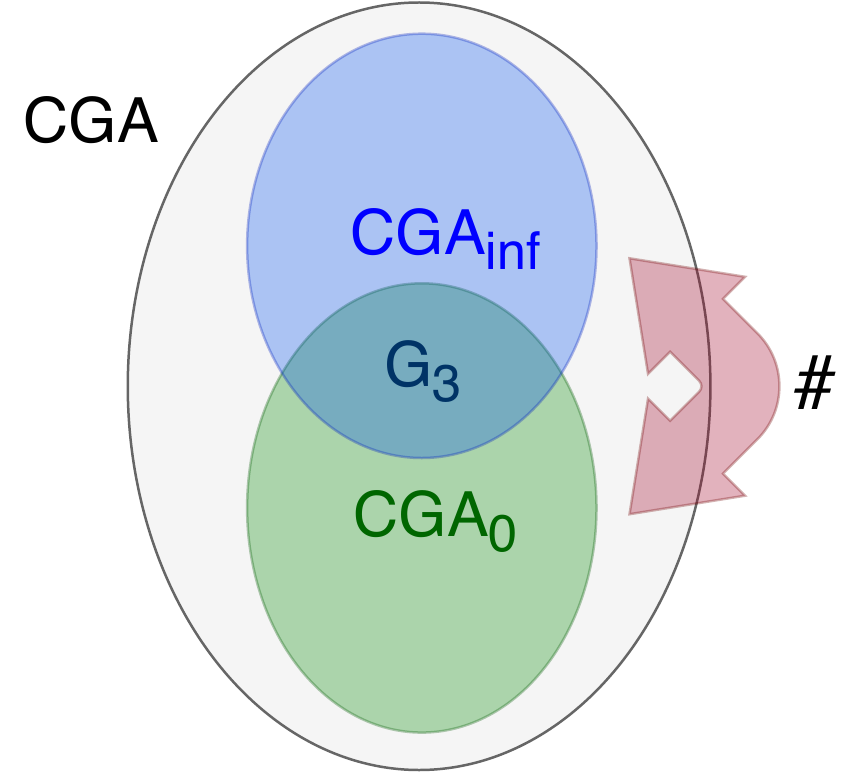}
 	\caption{Two copies of PGA inside CGA}
\end{figure}
We stress that $\sharp$ is an isomorphism of Clifford algebras and thus it preserves all products in the algebra.

Once we know how to identify PGA with any of the two subalgebras in CGA, it is easy to understand duality in PGA. It can be defined in the same way as the duality in CGA, by multiplication with a suitable pseudoscalar inverse. But we have to multiply by the inner product and also the pseudoscalar inverse must be taken with respect to the inner product. Such an inversion exists but it lies in the other subalgebra. Indeed, it is easy to see that $\bI\cdot\bI^\sharp=1$ and thus for  $\bI=\e_{0123}\in\CGA_0,$  the inner product inverse is $\bI^\sharp=\e_{123\infty}\in\CGA_\infty$ and vice versa: for $\bI^\sharp=\e_{123\infty}\in\CGA_\infty,$ the inner product inverse is $(\bI^\sharp)^\sharp=\bI.$
\begin{defn} \label{PGAdualityCGA}
Understanding PGA as a subalgebra of CGA, {\it projective duality} is an involution  defined for each $\bA\in \PGA$ by
\begin{align} \label{PGAduality}
\bA^{*_P}=(\bA \cdot \bI^\sharp)^\sharp = \bA^\sharp \cdot \bI.
\end{align}
\end{defn}
Note that the equality in the definition follows from the fact that $\sharp$ preserves inner product. 
In the diagram visualisation, this can be described by the commutative diagram in Figure \ref{PGAdualityDiagram}. The duality in PGA is the composition from left to right which is the same as the composition from right to left since it is an involution.
\begin{figure}[h] \label{PGAdualityDiagram}
	\includegraphics[scale=0.4]{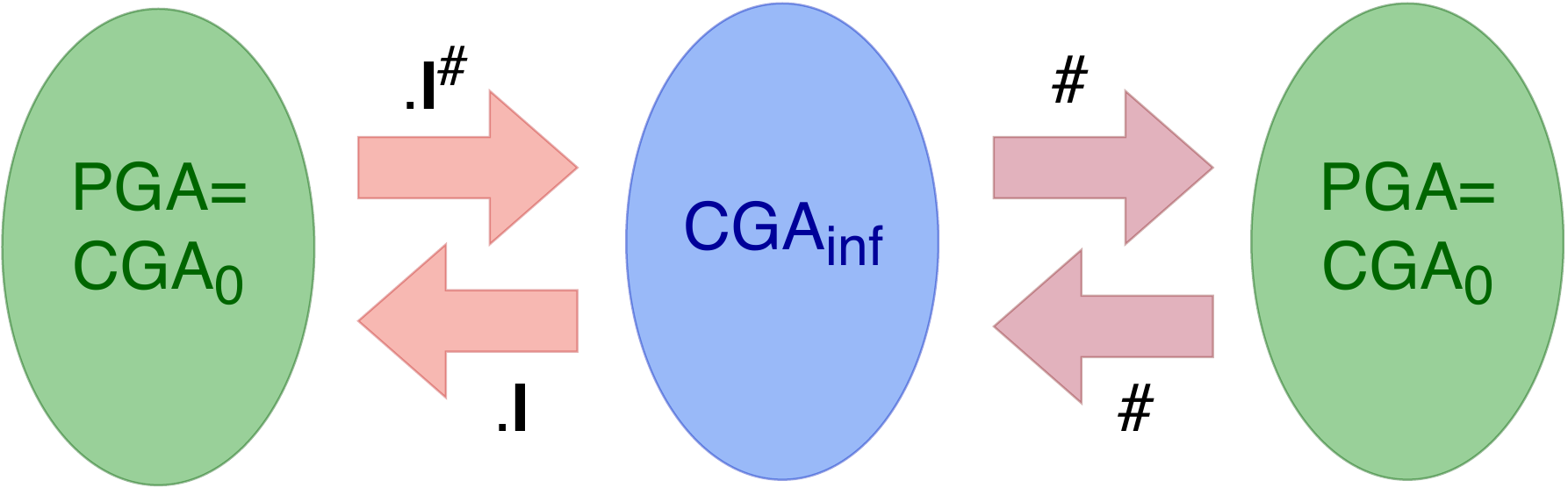}
	\caption{Duality in PGA viewed as a subalgebra of CGA}
\end{figure}

Having the duality in PGA defined according to the Definition \ref{PGAdualityCGA}, it is easy to relate it to the standard duality in CGA. Namely, it is a CGA duality twisted by a null vector as follows. 
\begin{prop} \label{PGAviaCGA}
	For $\bA\in\PGA$ the following identities hold
\begin{align}  \label{PGAduality2}
\bA^{*_P}=(\bA \wedge \e_\infty)^{*\sharp}=(\bA^\sharp \wedge \e_0)^*.
\end{align}
\end{prop}
\begin{proof}
By definition of PGA duality \eqref{PGAduality}, we need to show $(\bA \wedge \e_\infty)^{*}=\bA\cdot\bI^\sharp$ in order to prove the first identity. Indeed, we compute
\begin{align*}
(\bA \wedge \e_\infty)^{*}=(\bA \wedge \e_\infty)\cdot\bI_c^{-1}=\bA\cdot(\e_\infty\cdot\bI_c^{-1})=\bA\cdot\bI^\sharp.
\end{align*}
The first equality is the CGA duality in terms of the inner product, see \eqref{CGAduality}. The inner product coincides with the left contraction in this case since $\bI_c^{-1}$  is of the highest grade. Then the second equality is a general property of left contraction. The third equality  follows from the fact that $\e_\infty\cdot\bI_c^{-1}=\e_\infty\cdot\e_{0321\infty}=(\e_\infty\cdot\e_0)\e_{321\infty}=-\e_{321\infty}=\bI^\sharp$. The second identity in the proposition follows from the definition of isomorphism $\sharp$. Namely, 
$(\bA \wedge \e_\infty)^{\sharp}=-\bA^\sharp \wedge \e_0$ and $(\bI_c^{-1})^\sharp=-\bI_c^{-1}.$
\end{proof}

We also get a sort of PGA duality between the inner product and outer product similar to the duality between these products in CGA. Namely, supposing the grade of a blade $\bA$ is less or equal to $\bB$, we have
\begin{align} \label{PGAidentity}
	(\bA \wedge \bB)^{*_P}=\bA^\sharp \cdot \bB^{*_P}. 
\end{align}
This formula holds for general multivectors $\bA, \bB$ if we replace the inner product on the right--hand side by left contraction.
It is worth to look also at the relation between the duality in PGA and the usual Euclidean duality. For that we need to express a multivector $\bA\in\PGA$ as a sum $\bC_E+\bD_E\wedge\e_0,$ where $\bC_E,\bD_E\in\mathbb{G}_3.$ Then we compute
\begin{align} \label{PGAduality3}
\bA^{*_P}=(\bC_E+\bD_E\wedge\e_0)^{*_P}=-\bD_E^{*_E}+\bC_E^{*_E}\wedge \e_0.
\end{align} 
This formula is particularly convenient for an explicit computation of duals to the basis blades, see Table \ref{PGAdualitytable}.
\begin{table}[h] \label{PGAdualitytable}
	\begin{center}
		{\small\addtolength{\tabcolsep}{-3pt}
			\begin{tabular}{ |c|c|c|c|c|c|c|c|c|c|c|c|c|c|c|c|c| } 
				\hline
				& $1$ & $\e_0$ & $\e_1$ & $\e_2$ & $\e_3$  & $\e_{01}$ & $\e_{02}$ & $\e_{03}$ & $\e_{12}$ & $\e_{13}$ & $\e_{23}$ & $\e_{012}$ & $\e_{013}$ & $\e_{023}$ & $\e_{123}$ & I \\ 
				\hline
				$\bA$ &  a & b & c & d & e & f & g & h & i & j & k & l & m & n & o & p \\
				\hline
				$\bA^{*_P}$ &  p & o & -n & m & -l & -k & j & -i & -h & g & -f & -e & d & -c & b & a \\
				\hline
			\end{tabular}
		}
	\end{center}
	\caption{PGA duality}
\end{table}

\begin{rem} \label{remDual}
	The definition of isomorphism $\sharp$ is independent of the dimension.  However,
    for general dimension $n$ we have 
    \begin{align} \label{s}
    \bI\cdot\bI^\sharp=(-1)^{n(n+1)/2},
    \end{align}
    thus the inner product inversion to $\bI$ gains this sign and it should also enter in formula \eqref{PGAduality} for the definition of the PGA duality. Note that PGA duality is an involution or anti--involution depending on the dimension. Formulae \eqref{PGAduality2} and \eqref{PGAidentity} hold in any dimension. Indeed, the only change in the proof of Proposition \ref{PGAviaCGA} is that
    $
    \e_\infty\cdot\bI_c^{-1}
    =(-1)^{n(n+1)/2}\bI^\sharp.
    $
    The equation \eqref{PGAduality3} changes in $n$D to
    \begin{align*}
    (\bC_E+\bD_E\wedge\e_0)^{*_P}=(-1)^{n}\bD_E^{*_E}+\bC_E^{*_E}\wedge \e_0.
    \end{align*}
   
\end{rem}

\subsection{Geometry}
All geometric objects in PGA are constructed by means of regressive product which is dual to the outer product, c.f. \eqref{PGAline}, \eqref{PGAplane}. Let us start by computing the projective dual of a point. By applying formula \eqref{PGAduality} for the projective duality to the PGA point \eqref{PGApoint} we get
\begin{align} \label{dualPoint}
\bP^{*_P}=\e_0+\bP_E.
\end{align}
Hence a point in PGA is the dual to its usual homogeneous representation in $\mathbb{R}^4$. Now we compute the representation of PGA point in the subalgebra $\CGA_\infty$. Applying the isomorphism $\sharp$ to formula \eqref{PGApoint} for a point we get
\begin{align} \label{PGApointI}
\bP^\sharp
=\bI_E-\bP^{*_E}_E\e_\infty.
\end{align}
We will show in the following Proposition \ref{GeometricEmbedding} that this is a formula for the dual representation of a flat point in CGA. Indeed, looking at the direct representations of flat objects and Euclidean transformations in CGA we observe that they all  lie in the subalgebra $\CGA_\infty$. Hence this subalgebra is the right copy of PGA in CGA for geometric purposes and the map $\sharp$ gives a geometric embedding in the following sense.
\begin{prop} \label{GeometricEmbedding}
	 Let $\bP,\bell,\bp$ be representations of a point, a line and a plane in PGA, respectively. Then $\bP^\sharp,\bell^\sharp,\bp^\sharp$ are the dual (IPNS) representations of the same point (viewed as a flat point), line and plane in CGA. Moreover,
	if $\bV$ is a versor for a Euclidean transformation in PGA, then $\bV^\sharp$ is the versor for the same transformation in CGA.
\end{prop}

\begin{proof}
At first we prove the correspondence of points. We need to show that \eqref{PGApointI} is the form of the dual flat point $\bF\bP_c^*=(\bP_c \wedge\e_\infty)^*$. Indeed, we compute
\begin{align}
\bF\bP_c^*&=(\e_{0\infty}+\bP_E\e_\infty)^*=\bI_E-\bP^{*_E}_E \e_\infty=\bP^\sharp.
\end{align}
For the proof of the correspondence between other geometric objects we use the formula for dual point in terms of a conformal point. Since the dual point \eqref{dualPoint} actually is the conformal representation of a point without the quadratic part, we can write
 \begin{align} 
{\bP}^{*_P}
=\bP_c+(\e_0\cdot \bP_c)\e_\infty.
\end{align}
Now we use this formula to a PGA line  $\bell=\bP_1\vee \bP_2.$ 
By the definition of the regressive product and formula \eqref{PGAduality2} for the PGA duality we get
 \begin{align*}
\bell^\sharp&=
\left( (\bP_{1c}+(\e_0\cdot \bP_{1c})\e_\infty) \wedge (\bP_{2c}+(\e_0\cdot \bP_{2c})\e_\infty)\wedge\e_\infty\right)^{*}\\
&=(\bP_{1c}\wedge \bP_{2c} \wedge \e_\infty)^{*}=\bell_c^*.
\end{align*} 
Similarly, for a plane $\bp=\bP_1\vee \bP_2\vee \bP_3$ we get
  \begin{align*}
 \bp^\sharp& = (\bP_{1c}\wedge \bP_{2c} \wedge \bP_{3c} \wedge \e_\infty)^{*}=\bp_c^*,
 \end{align*} 
thus the representation of any object in PGA coincides with its dual representation in CGA. The claim about versors follows from the correspondence between objects and the fact that $\sharp$ respects geometric product and that it commutes with the reversion operation.
\end{proof}
\begin{rem}
	The duality in PGA is defined in such way that dual PGA points always coincide with their homogeneous representations, i.e. the formula \eqref{dualPoint} holds for arbitrary dimension. On the other hand, the correspondence in Proposition \ref{GeometricEmbedding}  changes sign according to the sign of \eqref{s}. Namely, in dimension $n$ we have
	 \begin{align*}
      \bP^\sharp=(-1)^{n(n+1)/2}(\bP_c \wedge \e_\infty)
     \end{align*}  
and also the remaining objects are endowed with the same sign.
\end{rem}
%

%
%
\subsection{Example}
\label{GanjaExample}
It is obvious that we do not need the quadratic parts of points in CGA when we deal with flat objects only. Projective algebra is certainly more efficient in that case. However, we do not need to abandon the CGA concept while using all the amazing formulae from PGA. We only have to keep in mind that points correspond to flat points and that the projective duality differs from the conformal duality, see \eqref{PGAduality} and \eqref{PGAduality2}.

Let us consider an elementary example to demonstrate a convenient way of using PGA inside CGA. We have a point $\bP$ lying on a line $\bell$ which intersects a plane $\bp$ and we have a sphere $\bs$ which also intersects the plane, see Figure \ref{RovinaKoule}.
\begin{figure}[h] \label{RovinaKoule}
	\includegraphics[scale=0.6]{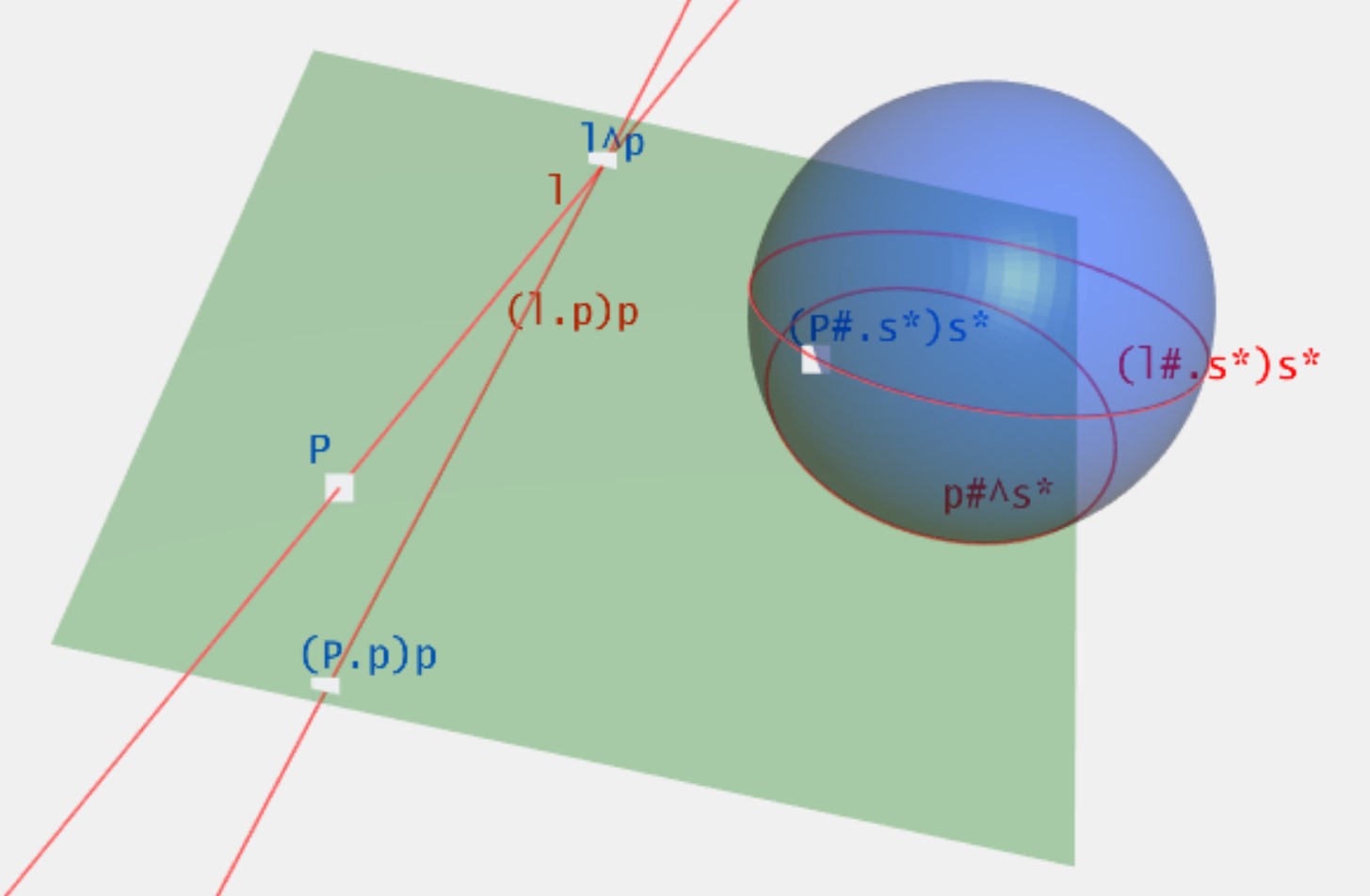}
	\caption{Intersections and projections of flats and rounds}
\end{figure}
Then we can use formulae from PGA to compute the intersection of the line and plane $\bell\wedge\bp$ and the orthogonal projection of the point to the plane $(\bP\cdot\bp)\bp$, the orthogonal projection of the line to the plane $(\bell\cdot\bp)\bp$. If we need to calculate with the sphere, we only need to replace $\e_0$ by $-\e_\infty$ in the representations of flat objects, which is realized by the map $\sharp$, and then we can use all formulae known in CGA. For instance, the intersection of the sphere and the plane is a dual circle $\bp^\sharp\wedge\bs^*$, the orthogonal projection of the point to the sphere is $(\bP^\sharp\cdot\bs^*)\bs^*$ and  the orthogonal projection of the line to the sphere is a circle $(\bell^\sharp\cdot\bs^*)\bs^*$. 

Figure \ref{RovinaKoule} was created by the web--based experimantal platform of Ganja.js \cite{ganja}. The corresponding full code follows.

\lstset{language=Java,basicstyle=\footnotesize,commentstyle=\color{OliveGreen},frame=single,keywordstyle=\color{blue},morekeywords={e1,e2,e3,e4,e5,no,ni}}
\begin{lstlisting}
// Create a Clifford Algebra with 4,1 metric for 3D CGA. 
Algebra(4,1,()=>{ 

// We start by defining a null basis, and upcasting for points
// which we need for rounds only
var ni = 1e4+1e5, no = .5e5-.5e4;
var up = (x)=> no + x + .5*x*x*ni;

// Sharp map in both directions in terms of CGA products  
var IN = (x)=> x + ni^(no<<x) + no^(no<<x);
var NI = (x)=> x + ni^(ni<<x) + no^(ni<<x);

// PGA pseudoscalar  
var I = no^1e1^1e2^1e3;

// PGA duality and upcasting to PGA  
var dual = (x)=> NI(x)<<I;
var upP = (x)=> dual(no+x);

// Definition of regressive product for 2 and three inputs  
var reg = (x,y)=>dual(dual(x)^dual(y));
var reg3 = (x,y,z)=>dual(dual(x)^dual(y)^dual(z));

// Formulas from PGA can be used in CGA

// We define 4 points
var P =  upP(0.5e1-1.5e3);
var P1 = upP(1e1), P2 = upP(1e2), P3 = upP(-1e3);

// Constructing line and plane from points
var l = ()=>reg(P,P2);
var p = ()=>reg3(P1,P2,P3);

// Intersection of line and plane
var Q = ()=>p^l;

// Projection of point and line to plane
var Pperp = ()=>(p<<P)*p;
var lperp = ()=>(p<<l)*p;

// If we want we can also intersect with a sphere or project on it
var s = ()=>up(0.5e2+1e1)-0.5*0.25*ni;

var c = ()=>s^NI(p);
var cperp = ()=>(s<<NI(l))*s;

// Graph the items as CGA elements A -> !NI(A)
document.body.appendChild(this.graph([
0xE0880000, !NI(P), "P",  // point
0xE0880000, !NI(Q), "l^p",  // point
0xE0880000, !NI(Pperp), "(P.p)p",  // point
0xE0000000,!NI(lperp) , "(l.p)p",  // line
0xE00000FF,!cperp , "(l.s*)s*",  // circ
0xE00000FF,!c , "p^s*",  // circ
0xE0000000, !NI(l), "l",  // line
0xE0008800, !NI(p), "p",  // plane
0xE00000FF, !s, "s"   // sphere
],{conformal:true,gl:true,grid:false})); 
});
\end{lstlisting}

\section{Conclusion}
We introduced a notation for PGA objects that is compatible with their CGA description, once PGA is understood as an subalgebra of CGA. We also solved the issues with the noninvertibility of the PGA pseudoscalar in computing duality and showed the exact forms of dual counterparts to geometric primitives in PGA using another copy of PGA in CGA. This has great potential for symbolic calculations and their software implementation, because we can flexibly switch between PGA notation which is efficient for flat objects manipulation and CGA operations on round elements in our problems. The next step is to show the advantages of this approach in various computational platforms, together with code optimisation and applications.

\bibliographystyle{plain}

\end{document}